\documentclass[paper=a4, fontsize=10pt]{amsart} 

\usepackage{mathalfa,amsmath,amsthm,amsfonts,amssymb}
\usepackage{pgfpages}
\usepackage{mathdots}
\usepackage{marginnote}
\usepackage{youngtab}
\usepackage{framed,lambda}
\usepackage[english]{babel}
\usepackage{cleveref,enumitem,imakeidx}
\usepackage{color,xcolor,graphicx,tikz,tikz-cd}
\usepackage{stmaryrd}
\usepackage[all]{xy}
\theoremstyle{plain}
\newtheorem{lemma}[subsection]{Lemma}
\newtheorem{theorem}[subsection]{Theorem}
\newtheorem{prop}[subsection]{Proposition} 
\newtheorem{cor}[subsection]{Corollary}

\theoremstyle{definition}

\theoremstyle{remark}
\newtheorem{example}[subsection]{Example}

\numberwithin{equation}{section} 
\numberwithin{figure}{section} 
\numberwithin{table}{section} 
\numberwithin{subsection}{section} 

\makeatletter
\let\c@equation\c@subsection
\makeatother

\makeatletter
\let\c@figure\c@subsection
\makeatother

\def\tensor{\mathop{\otimes}}
\def\mid{\,\middle\vert\,}
\newif\ifdraft

\drafttrue
\draftfalse
\ifdraft
\pgfpagesuselayout{resize to}[a3paper]
\pagecolor[rgb]{0,0,0}
\color[rgb]{0.0,0.9,0.4}
\fi
\draftfalse

\def\borel{\ensuremath{B}}
\def\con{\ensuremath {N^*}}
\def\torus{\ensuremath{T}}

\def\N{\ensuremath{N\left(\mathbf k[t,t^{-1}]\right)}}
\def\Borel{\ensuremath{\mathcal B}}
\def\Borelhat{\ensuremath{\mathcal B}}
\def\cartan{\ensuremath{\mathfrak h}}

\def\simple{\ensuremath{\Simple_0}}
\def\Simple{\ensuremath S}
\def\roots{\ensuremath{\Delta_0}}
\def\proots{\ensuremath{\Delta_0^+}}
\def\Roots{\ensuremath{\Delta}}
\def\ord{\ensuremath{\operatorname{ord}}}
\def\gl{\ensuremath{G}}
\def\lgl{\ensuremath{\mathfrak g}}
\def\gL{\ensuremath{L\gl}}
\def\gLhat{\gL}
\def\gOhat{\ensuremath{L^+\gl}}
\def\gO{\gOhat}
\def\para{\ensuremath{P}}
\def\Para{\ensuremath{\mathcal P}}
\def\u{\ensuremath{\mathfrak u}}
\def\e{\ensuremath e}
\def\Y{\ensuremath{Y}}
\def\W{\ensuremath {W}}
\def\E{\ensuremath{E}}
\def\Z{\ensuremath{Z}}
\def\What{\ensuremath{\widehat\W}}
\def\coroots{\ensuremath{Q}}

\def\betac{\ensuremath{q}}
\def\w{\ensuremath{\varpi}}

\def\image{\ensuremath{\operatorname{Im}}}
\def\fij#1.#2.{\ensuremath{\chi(#2,#1)}}
\def\f#1.{\ensuremath{F^i_{#1}}}  
\def\row{\ensuremath{\mathcal Row}}
\def\red{\ensuremath{\mathcal Red}}
\def\blue{\ensuremath{\mathcal Blue}}
\def\set{\ensuremath{\mathcal{S}}}
\def\Lambda{\ensuremath{d}}

\def\Ni{\ensuremath{\mathcal N}}

\def\mod{\ensuremath{\mathrm{mod}\,}}
\def\dynkAnHat{\vskip -3ex$$\xymatrix{ \overset{\alpha_1}\bullet \ar@{-}[r] &\overset{\alpha_2}\bullet \ar@{-}[r] & \cdots\ar@{-}[r] & \overset{\alpha_{n-2}}\bullet \ar@{-}[r] &\overset{\alpha_{n-1}}\bullet\ar@{-}[dll]\\ &&\overset{\alpha_0}\bullet \ar@{-}[ull] & & }$$}
\def\dynkAn{$$\xymatrix{\overset{\alpha_1} \bullet \ar@{-}[r] & \overset{\alpha_2}\bullet \ar@{-}[r] & \cdots \ar@{-}[r]&\overset{\alpha_{n-2}}\bullet \ar@{-}[r] &\overset{\alpha_{n-1}} \bullet}$$}

\title{Cotangent Bundles of Partial Flag Varieties and Conormal Varieties of their Schubert Divisors}

\author{V. Lakshmibai}
\author{Rahul Singh} 

\date{\normalsize\today} 

\begin{document}

\maketitle 

\begin{abstract}
Let $P$ be a parabolic subgroup in $G=SL_n(\mathbf k)$, for $\mathbf k$ an algebraically closed field.
We show that there is a $G$-stable closed subvariety of an affine Schubert variety in an affine partial flag variety which is a natural compactification of the cotangent bundle $T^*G/P$.
Restricting this identification to the conormal variety $N^*X(w)$ of a Schubert divisor $X(w)$ in $G/P$, we show that there is a compactification of $N^*X(w)$ as an affine Schubert variety.
It follows that $N^*X(w)$ is normal, Cohen-Macaulay, and Frobenius split.
\end{abstract}

\section{Introduction}
Let the base field $\mathbf k$ be algebraically closed. 
Consider a cyclic quiver with $h$ vertices and
dimension vector ${\underline d}=(d_1,\cdots, d_h)$.
Let
$$Rep(\underline d, \widehat A_h)=\operatorname{Hom}(V_1,V_{2})\times \cdots\times \operatorname{Hom}(V_h,V_{1}),\ 
GL_{{\underline d}}=\prod_{1\le i\le h}\,GL{(V_i)}$$ 
We have a natural action of $GL_{{\underline d}}$ on $Rep(\underline d,\widehat A_h)$: for
$g=(g_1,\cdots,g_h)\in GL_{{\underline d}}$ and $f=(f_1,\cdots,f_h)\in
Rep(\underline d,\widehat A_h)$,
$$g\cdot f=(g_2f_1g_1^{-1},g_3f_2g_2^{-1},\cdots,g_1f_hg_h^{-1})$$
Let $$Z=\{(f_1,\cdots,f_h)\in Z\,|\,f_h\circ
f_{h-1}\circ\cdots\circ f_1:V_1\rightarrow V_1{\mathrm{\ is\
nilpotent}}\}$$

Clearly $Z$ is $GL_{{\underline d}}$-stable. 
Set $n=\sum d_i$, and let $\widehat{SL}_n$ be the affine type Kac-Moody group with Dynkin diagram $\widehat A_{n-1}$.
Lusztig (cf.\cite{gl}) has shown that an orbit closure in $Z$ is canonically isomorphic to an open subset of a Schubert variety in ${\widehat{SL}}_n/Q$, where $Q$ is the parabolic subgroup of ${\widehat{SL}}_n$ corresponding to omitting the simple roots $\alpha_0,\alpha_{d_{1}}, \alpha_{d_{1}+d_{2}},\cdots,\alpha_{d_{1}+\cdots +d_{h-1}}$.
\\
\\
Let now $h=2$ and consider the subvariety $Z_0$ of $Z$ given by
$$Z_0=\{(f_1,f_2)\in \operatorname{Hom}(V_1,V_2)\times\operatorname{Hom}(V_2,V_1)\,|\,f_2\circ f_1=0,f_1\circ f_2=0\}$$
with the dimension vector $(d_1,d_2)$.
Strickland (cf. \cite{strickland}) has shown that the irreducible components of $Z_0$ give the conormal varieties of the determinantal varieties in $\operatorname{Hom}(V_1,V_2)=\operatorname{Mat}_{d_2,d_1}(\mathbf k)$, the set of $d_2\times d_1$ matrices with entries in $\mathbf k$.
%
A determinantal variety in $\operatorname{Hom}(V_1,V_2)$ being canonically isomorphic to an open subset in a certain
Schubert variety in $Gr_{d_1,d_1+d_2}$ (the Grassmannian variety of $d_1$-dimensional subspaces of $\mathbf k^{d_1+d_2}$) (cf.\cite{gp2}),
the above two results of Lusztig and Strickland suggest a connection between conormal varieties to Schubert varieties in the (finite-dimensional) flag variety and affine Schubert varieties. 
\\
\\
Let $\gl=SL_n(\mathbf k)$.
We view the loop group $\gLhat=SL_n\left(\mathbf k[t,t^{-1}]\right)$ as a Kac-Moody group of type $\widehat A_{n-1}$.
Let $P$ be the parabolic subgroup of $G$ corresponding to omitting the simple root $\alpha_d$ for some $1\leq d\leq n-1$, and $\Para$ the parabolic subgroup of $\gLhat$ corresponding to omitting the simple roots $\alpha_0,\alpha_d$.
Lakshmibai (\cite{vl}) has shown that the cotangent bundle $T^*G/P$ is an open subset of a Schubert variety in $\gLhat/\Para$.
Let $w_0$ be the longest element in the Weyl group $W$ of $G$.
In \cite{rv}, we have shown that the conormal variety of the Schubert variety $X_P(w)$ is an open subset of a Schubert variety in $\gLhat/\Para$ if and only if the Schubert variety $X_P(w_0w)$ is smooth.
\\
\\
The approach adopted in \cite{vl,rv} seems to be quite successful in relating cotangent bundles and conormal varieties of classical Schubert varieties to affine Schubert varieties.
In this paper, for any parabolic subgroup $P$ of $G$, we first construct an embedding of the cotangent bundle $T^*G/P$ inside an affine Schubert variety $X(\kappa)$ (\Cref{injective}).
We then show that the conormal variety of a Schubert divisor in $G/P$ is an open subset of some affine Schubert subvariety of $X(\kappa)$ (\Cref{vk}).
In particular, we obtain that the conormal variety of a Schubert divisor is normal, Cohen-Macaulay, and Frobenius split (\Cref{frob}).
\\
\\
\newif\ifminuscule
\minusculefalse
Let $\gO=G\left(\mathbf k[t]\right)$ and let \Borel\ be a Borel subgroup of \gLhat\ contained in \gO. 
The action of of \gl\ on $V=\mathbf k^n$ defines an action of $\gLhat$ on $V[t,t^{-1}]$. 
The affine Grassmannian $\gLhat/\gOhat$ is an ind-variety whose points are the lattices with virtual dimension $0$ (see \Cref{vdim} for the definition of a lattice and its virtual dimension).
The affine flag variety $\gLhat/\Borel$ is an ind-variety whose points are affine flags $L_0\subset L_1\subset\ldots\subset L_{n-1}$ satisfying the incidence relation $tL_{n-1}\subset L_0$.
\\
\\
Let \Ni\ be the variety of $n\times n$ nilpotent matrices on which \gl\ acts by conjugation. 
In \cite{gl:green}, Lusztig constructs a \gl-equivariant embedding $\psi:\Ni\hookrightarrow\gLhat/\gO$ which takes \gl-orbit closures in \Ni\ to \gl-stable Schubert varieties in $\gLhat/\gO$.
Let $\boldsymbol\theta$ be the Springer resolution $T^*G/B\rightarrow\Ni$ given by $(g,X)\mapsto gXg^{-1}$.
In \cite{crv} Lakshmibai et al. construct a lift of $\psi$, namely a map $\eta$ embedding $T^*\gl/\borel$ into the affine flag variety $\gLhat/\Borel$, leading to the following commutative diagram:\begin{center}
\includegraphics{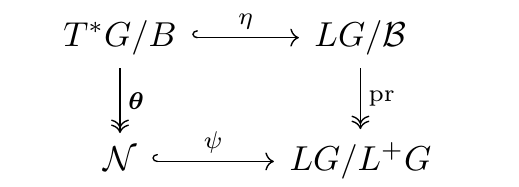}
\end{center}
where $\operatorname{pr}$ is the projection map. 
Under the identifications of the previous paragraph, it takes the affine flag $L_0\subset\ldots\subset L_{n-1}$ to the lattice $L_0$.
In this paper, for any parabolic subgroup $P$ of $G$, we define a generalization $\phi_P$, as described below,
of the map $\eta$, and then study the conormal variety of a Schubert divisor by identifying its image under $\phi_P$.
\\
\\
Let $P$ be a parabolic subgroup of $G$ corresponding to some subset $S_P\subset S_0$. 
We denote by \Para\ he parabolic subgroup of \gLhat\ corresponding to $S_P\subset S=\simple\sqcup\{\alpha_0\}$.
Let \u\ be the Lie algebra of the unipotent radical of $P$.
Using the identification  $T^*\gl/\para=\gl\times^\para\u$ (see \Cref{defCotan}), we define in \Cref{defnPhi} the map $\phi_\para:T^*G/P\hookrightarrow\mathcal G/\Para$ by $\phi_\para(g,X)=g\left(1-t^{-1}X\right)(\mathrm{mod}\,\Para)$, which sits in the following commutative diagram (see \Cref{springer}):
\begin{center}
\includegraphics{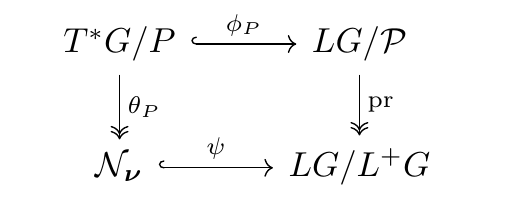}
\end{center}
\vspace{-2pt}
The variety $\Ni_{\boldsymbol\nu}$ is the closure of a \gl-orbit in \Ni. 
The map $\theta_P:T^*G/P\rightarrow\Ni_{\boldsymbol\nu}$, given by $(g,X)\mapsto gXg^{-1}$ is a resolution of $\Ni_{\boldsymbol\nu}$ (cf. \cite{fu}), and $\operatorname{pr}$ is the quotient map.
The closure of $\image(\phi_\para)$ in $X_\Para(\kappa)$ is a \gl-stable compactification of $T^*\gl/\para$.
We also identify the $\kappa\in\What^\Para$ (see \Cref{formula:kappa} and \Cref{injective}) for which $\image{\phi_\para}\subset X_\Para(\kappa)$, and further, $\kappa$ is minimal for this property.
Finally, given a Schubert divisor $X_P(w)\subset G/P$, we identify a Schubert variety $X_\Para(v)$ which is a compactification of the conormal variety $N^* X_P(w)$ (see \Cref{vk}).
\\
\\
When \para\ is maximal, $\phi_\para$ is the same as the map in \cite{vl}, and gives a dense embedding of $T^*G/P$ , i.e., $X_\Para(\kappa)$ is a compactification of $T^*\gl/\para$.
Mirkovi\'c and Vybornov \cite{mv} have constructed another lift of Lusztig's embedding along the Springer resolution, which we briefly discuss in \Cref{mv}. 
Their lift is in general different, but agrees with $\phi_P$ exactly when $P$ is maximal.
The reader is cautioned that Mirkovi\'c and Vybornov work with a different choice of Lusztig's map given by $X\mapsto (1-t^{-1}X)^{-1}(\mod\gOhat)$.
This difference is simply a matter of convention. 
\\
\\
The paper is organized as follows.
In \S2, we discuss some generalities on loop groups, and the associated root system and Weyl group.
In \S3, we define the embedding $\phi_P:T^*G/P\hookrightarrow\mathcal G/\Para$, relate it to Lusztig's embedding $\psi:\mathcal N\rightarrow\mathcal G/\gOhat$ via the Springer resolution, and discuss the relationship between $\phi_P$ and the map $\widetilde\psi$ of \cite{mv}.
In \S4, we identify the minimal $\kappa\in\What$ for which $\phi_P(T^*G/P)\subset X_\Para(\kappa)$.
In particular, $\phi_P$ identifies a \gl-stable subvariety of $X_\Para(\kappa)$ as a compactification of $T^*G/P$.
We also compute the dimension of $X_\Para(\kappa)$.
This allows us to recover the result of \cite{vl} and further show that $X_\Para(\kappa)$ is a compactification of $T^*G/P$ if and only if $P$ is a maximal parabolic subgroup.
Finally, in \S6, we apply $\phi_P$ to the conormal variety $N^* X_P(w)$ of a Schubert divisor $X_P(w)$ in $G/P$ to show that $N^* X_P(w)$ can be embedded as an open subset of an affine Schubert variety.

\def\gln #1.{\ensuremath{SL_n\left(\mathbf #1\right)}}
\section{The Loop Group}
Let $\mathbf k$ be an algebraically closed field.
In this section, we discuss some standard results on the root system, Weyl groups, and the Bruhat decomposition of the loop group $\gLhat=SL_n\left(\mathbf k[t,t^{-1}]\right)$.
For further details, the reader may refer to \cite{vk,sk,remy}.

\subsection{The Loop Group and its Dynkin Diagram}
Let \gl\ be the special linear group $\gln k .$.
The subgroup \borel\ (resp. $\borel^-$) of upper (resp. lower) triangular matrices in \gl\ is a Borel subgroup of \gl, and the subgroup \torus\ of diagonal matrices is a maximal torus in \gl. 
The root system of \gl\ with respect to $(\borel,\torus)$ has Dynkin diagram $A_{n-1}$.
We use the standard labeling of simple roots \simple\ 
\dynkAn
Let $\roots$ be the set of roots and $\proots$ the set of positive roots.
Let $L^\pm\gl=\gln k\left[t^{\pm 1}\right].$  and $\pi_\pm:L^\pm\gl\rightarrow\gl$ the surjective map given by $t^{\pm 1}\mapsto 0$. 
Then $\Borel^\pm:=\pi_\pm^{-1}(\borel)$ are \emph{opposite} Borel subgroups of \gLhat. 
The loop group \gL\ is a (minimal) Kac-Moody group (cf. \cite{remy}) with torus $T$ and Dynkin diagram $\widehat A_{n-1}$.
\dynkAnHat
We denote the Borel subgroup $\Borel^+$ as just \Borel.
Parabolic subgroups containing \borel\ (resp. \Borel) in \gl\ (resp. \gLhat) correspond to subsets of \simple\ (resp. \Simple). 
For \para\ the parabolic subgroup in \gl\ corresponding to $S_\para\subset\simple$, the parabolic subgroup $\Para\subset\gLhat$ corresponding to $S_\para\subset\Simple$ is given by $\Para=\pi^{-1}(\para)$. 
In particular, the subgroup $\gOhat$ is the parabolic subgroup corresponding to $\simple\subset\Simple$.


\subsection{The Root system of $A_{n-1}$}
\label{rootSystemA}
Consider the vector space of $n\times n$ diagonal matrices, with basis $\left\{\E_{i,i}\mid 1\leq i\leq n\right\}$.
Writing $\left\langle\epsilon_i\mid 1\leq i\leq n\right\rangle$ for the basis dual to $\left\{\E_{i,i}\mid 1\leq i\leq n\right\}$, we can identify the set of roots $\roots$ of \gl\ as follows:
$$\roots=\left\{\langle\epsilon_i-\epsilon_j\rangle\mid 1\leq i\neq j\leq n\right\}$$
The positive root $\alpha_i+\ldots+\alpha_{j-1}$, where $i<j$, corresponds to $\epsilon_i-\epsilon_j$, and the negative root $-(\alpha_i+\ldots+\alpha_{j-1})$ corresponds to $\epsilon_j-\epsilon_i$.
For convenience, we shall denote the root $\epsilon_i-\epsilon_j$ by $(i,j)$, where $1\leq i\neq j\leq n$.
The action of $W(\cong S_n)$ on $\roots$ is given by $w(i,j)=(w(i),w(j))$.

\subsection{The Root system of $\widehat A_{n-1}$}
Let \Roots\ be the root system associated to the Dynkin diagram $\widehat A_{n-1}$.
We have 
$$\Roots=\left\{n\delta+\alpha\mid n\in\mathbb Z,\alpha\in\roots\right\}\bigsqcup\left\{n\delta\mid n\in\mathbb Z, n\neq0\right\}$$ 
where $\delta=\alpha_0+\theta$ is the basic imaginary root in \Roots, and $\theta=\alpha_1+\cdots+\alpha_{n-1}$ is the highest root in \proots.
The set of positive roots $\Delta^+$ has the following description:
$$\Delta^+=\{n\delta+\alpha\vert\, n>0,\alpha\in\roots\}\bigsqcup\proots\bigsqcup\{n\delta\vert\,n>0\}$$
Let $\sim$ be the equivalence relation on $\mathbb Z\times\mathbb Z$ given by \begin{align*}
\hspace{40pt}    (i,j)&\sim(i+kn,j+kn)&k\in\mathbb Z
\end{align*}
We identify \Roots\ with $\mathbb Z\times\mathbb Z/\sim$ as follows:\begin{align}\label{eq:zz}
\hspace{80pt}   k\delta &\mapsto(0,kn)  &   \\
    (i,j)+k\delta       &\mapsto(i,j+kn)&\text{for }(i,j)\in\roots\nonumber
\end{align}
Under this identification, $(i,j)\in\Roots^+$ if and only if $i<j$.

\subsection{The Weyl group}
\label{permPres}
Let $N$ be the normalizer of $T$ in \gl. 
We identify the Weyl group \What\ associated to $\widehat A_{n-1}$ with $N(\mathbf k[t,t^{-1}])/T$. 
Let $\E_{i,j}$ be the $n\times n$ matrix with $1$ in the $(i,j)$ position, and $0$ elsewhere.
The elements of $N\left(\mathbf k[t,t^{-1}]\right)$ are matrices of the form $\sum\limits_{1\leq i\leq n}t_i\E_{\sigma(i),i}$, where \begin{itemize}
    \item   $(t_i)$ is a collection of non-zero monomials in $t$ and $t^{-1}$.
    \item   $\sigma$ is a permutation of $\{1\ldots n\}$. 
    \item   $\operatorname{det}\left(\sum\limits_{i=1}^n t_i\E_{\sigma(i),i}\right)=1$.
\end{itemize}
Consider the homomorphism 
$N\left(\mathbf k[t,t^{-1}]\right)\longrightarrow GL_n\left(\mathbf k[t,t^{-1}]\right)$ given by \begin{align}\label{apm}
    \sum\limits_{i=1}^n t_i\E_{\sigma(i),i} \longmapsto\sum\limits_{i=1}^n t^{\ord(t_i)}\E_{\sigma(i),i}
\end{align}
The kernel of this map is \torus, and so we can identify \What\ with the group of $n\times n$ permutation matrices $M$, with each non-zero entry a power of $t$, and $\ord(\det M)=0$.
For $w\in\What$, we call the matrix corresponding to $w$ the \emph{affine permutation matrix} of $w$.

\subsection{Generators for \What}
\label{ltau}
We shall work with the set of generators for $\What$ given by $\{s_0, s_1,\ldots, s_{n-1}\}$, where $s_i, 0\le i\le n-1$ are the reflections with respect to $\alpha_i, 0\le i\le n-1$. 
Note that $\left\{\alpha_i\mid 1\le i\le n-1\right\}$ being the set of simple roots of $G$, the Weyl group $W$ of $G$ is the subgroup of \What\ generated by $s_1,\ldots,s_{n-1}$.
The affine permutation matrix of $w\in W$ is $\sum\E_{w(i),i}$.
The affine permutation matrix of $s_0$ is given by
$$\begin{pmatrix}
0&0&\cdots & t^{-1}\\
0&1&\cdots &0\\
\vdots & \vdots & \vdots & \vdots\\
0&\cdots &1 &0\\
t &0&0&0
\end{pmatrix}
$$
For $1\leq a<b\leq n$, the reflection with respect to the positive root $(a,b)$ is given by the affine permutation matrix \begin{align*}
    s_{(a,b)}&=\E_{a,b}+\E_{b,a}+\sum\limits_{\substack{1\leq i\leq n\\ i\neq a,b}} \E_{i,i} 
\end{align*}

\subsection{Decomposition of \What\ as a semi-direct product}
\label{decomposition}
Consider the element $s_\theta\in W$, reflection with respect to the root $\theta$. 
There exists (cf. \cite{sk},\S13.1.6) a group isomorphism $\widehat{W}\rightarrow W\ltimes\coroots$ given by \begin{align*}\hspace{30pt}    
\hspace{30pt}   s_i &\mapsto (s_i,0)                    &\text{ for }1\leq i\leq n-1\\
                s_0 &\mapsto (s_\theta,-\theta^{\vee})  &
\end{align*} 
where \coroots\ is the coroot lattice 
and $\theta^\vee=\alpha_1^\vee+\cdots+\alpha_{n-1}^\vee$.
The simple coroot $\alpha_i^\vee\in\cartan,\, 1\leq i\leq n-1$, is given by the matrix $\E_{i,i}-\E_{i+1,i+1}$.
As shown in \cite{crv}, Section 2.4, a lift to $N(\mathbf k[t,t^{-1}])$ of $(\operatorname{id},\alpha_i^\vee)\in\What$, for $1\leq i\leq n-1$ is given by 
$$\tau_{\alpha_i^\vee}=\sum\limits_{k\neq i,i+1}\E_{k,k}+t^{-1}E_{i,i}+tE_{i,i}$$
For $q\in\coroots$, we will write $\tau_q$ for the image of $(\operatorname{id},q)\in W\ltimes\coroots$ in $\widehat W$.
Observe that for $\alpha\in\roots,\,q\in\coroots$ such that $\alpha=(a,b)$ and $\tau_q=\sum t_i\E_{ii}$, we have
\begin{align}\label{actionOfq}
    \alpha(q)=\ord(t_b)-\ord(t_a)
\end{align}
The action of $\tau_q$ on $\Delta$ is given by $\tau_q(\delta)=\delta$ and \begin{align}
\label{actionOfTau}
\hspace{68pt}   \tau_q(\alpha)  &=\alpha-\alpha(q)\delta&\text{ for }\alpha\in\roots
\end{align} 
In particular, for $\alpha\in\proots$, $\tau_q(\alpha)>0$ if and only if $\alpha(q)\leq 0$.
\begin{cor}
\label{count}
For $\alpha\in\proots$, $\tau_qs_\alpha>\tau_q$ if and only if $\alpha(q)\leq 0$ and $s_\alpha\tau_q>\tau_q$ if and only if $\alpha(q)\geq 0$ .
\end{cor}
\begin{proof}
Follows from the equivalences (cf. \cite{sk}) $ws_\alpha>w\iff w(\alpha)>0$ and $s_\alpha w>w\iff w^{-1}(\alpha)>0$ applied to $w=\tau_q$.
\end{proof}

\subsection{The Coxeter Length}
\label{formula:length}
Let $w\in\What$ be given by the affine permutation matrix $w=t^{c_i}E_{\sigma(i),i}$, i.e. $w(i)=\sigma(i)-c_i n$. 
We have the formula:\begin{align}\label{eq:length} 
    l(w)=\sum\limits_{1\leq i<j\leq n} \left\lvert c_i-c_j-f_\sigma(i,j)\right\rvert
\end{align}
where $f_\sigma(i,j)=\begin{cases} 0&\text{ if }\sigma(i)<\sigma(j)\\1&\text{ otherwise}\end{cases}$.
\begin{proof}
We use \Cref{eq:zz} to write $\Roots^+=\left\{(i,j)\mid 1\leq i\leq n,\,i<j\right\}$.
We know from \cite{sk} that $l(w)=\#\left\{\alpha\in\Roots^+\mid w(\alpha)<0\right\}$                 
\begin{align*}
        &=\#\left\{(i,j)\mid 1\leq i\leq n,\,i<j,\,w(i)>w(j)\right\}        \\
        &=\sum\limits_{1\leq i\leq n}\sum\limits_{1\leq j\leq n}\#\left\{(i,j')\mid j'=j\ \mod n,\,i<j',\,w(i)>w(j')\right\} \\
        &=\sum\limits_{1\leq i<j\leq n}\#\left\{(i,j+kn)\mid k\geq 0,\,w(i)>w(j+kn)\right\}+\#\left\{(j,i+kn)\mid k\geq 1,\,w(j)>w(i+kn)\right\}\\
        &=\sum\limits_{1\leq i<j\leq n}\#\left\{k\mid 0\leq k<\frac{\sigma(i)-\sigma(j)}{n} +c_j-c_i\right\}+\#\left\{k\mid 1\leq k<\frac{\sigma(j)-\sigma(i)}{n}-c_j+c_i\right\}\\
        &=\sum\limits_{1\leq i<j\leq n} \left\lvert c_i-c_j-f_\sigma(i,j)\right\rvert
\end{align*}
\end{proof}

\subsection{The Bruhat Decompostion}
The Bruhat decomposition of $\gL$ is given by $\gL=\bigsqcup\limits_{w\in\What}\Borel w\Borel$.
For $w\in\What$, let $X_\Borel(w)\subset\gL/\Borel$ be the {\em affine Schubert variety}:\vskip -4pt
$$    X_\Borel(w)=\overline{\Borel w\Borel}(\mod\Borel)=\bigsqcup\limits_{v\leq w}\Borel v\Borel(\mod\Borel)$$
The Bruhat order $\leq$ on \What\ reflects inclusion of Schubert varieties, i.e., $v\leq w$ if and only if $X_\Borel(v)\subseteq X_\Borel(w)$. 
\\
\\
Let $\What^\Para\subset\What$ be the set of minimal (in the Bruhat order) representatives of $\What/\What_\Para$.
The Bruhat decomposition with respect to \Para\ is given by $$\gLhat=\bigsqcup\limits_{w\in\What^\Para}\Borelhat w\Para$$
For $w\in\What^\Para$, the {\em affine Schubert variety} $X_\Para(w)\subset\gLhat/\Para$ is given by
$$    X_\Para(w)=\overline{\Borelhat w\Para}(\mod\Para)=\bigsqcup\limits_{\substack{v\leq w\\ v,w\in\What^\Para}}\Borelhat w\Para(\mod\Para)$$
For $w\in\What^\Para$, $X_\Para(w)$ is a normal (cf. \cite{gf}), projective variety of dimension $l(w)$, the length of $w$. 

\begin{prop}
\label{table}
Suppose $w\in\What$ is given by the affine permutation matrix 
$$\sum t_i\E_{\sigma(i),i}=\sum \E_{\sigma(i),i}\sum t_i\E_{i,i}=\sigma\tau_q$$ 
with $\sigma\in S_n$ and $\tau_q=\sum t_i\E_{i,i}\in\coroots$.
Let $1\leq a<b\leq n$ and set $s_r=s_{(a,b)}$, $s_l=\sigma s_r\sigma^{-1}=s_{(\sigma(a),\sigma(b))}$. 
\begin{description}
    \item[Case 1] Suppose $\ord(t_a)=\ord(t_b)$. 
        Then $s_lw=ws_r$. 
        The set $\{w,s_lw\}$ has a unique minimal element $u$, given by $u=\begin{cases} w&\text{if }\ \sigma(a)<\sigma(b)\\ s_lw&\text{if }\ \sigma(a)>\sigma(b)\end{cases}$. 
    \item[Case 2] Suppose $\ord(t_a)\neq\ord(t_b)$. 
        Then $s_lw\neq ws_r$, and the set $\{w,s_lw,ws_r,s_lws_r\}$ has a unique minimal element $u$, given by $u=\begin{cases} s_lws_r&\text{if }\ \ord(t_a)<\ord(t_b)\\ w&\text{if }\ \ord(t_a)>\ord(t_b)\end{cases}$.
        Further, we have $u<s_lu<s_lus_r$ and $u<us_r<s_lus_r$. 
\end{description}
\end{prop}
\begin{proof}
Recall from \Cref{actionOfq} that $q(a,b)=\ord(t_b)-\ord(t_a)$. 
Suppose first that $\ord(t_a)=\ord(t_b)$, which from \Cref{actionOfTau} is equivalent to $\tau_q(a,b)=(a,b)$.
It follows that 
$$ws_r=\sigma\tau_q s_{(a,b)}=\sigma s_{(a,b)}\tau_q= s_{(\sigma(a),\sigma(b)}\sigma\tau_q=s_lw$$
Recall that for $\alpha\in\proots$, we have $ws_\alpha>w$ if and only if $w(\alpha)>0$.
Now, $w(a,b)=\sigma\tau_q(a,b)=\sigma(a,b)=(\sigma(a),\sigma(b))$ being positive if and only if $\sigma(a)<\sigma(b)$, it follows that $w<ws_r$ if and only if $\sigma(a)<\sigma(b)$.
\\
\\
Suppose now that $\ord(t_a)\neq\ord(t_b)$, and set $k=\ord(t_a)-\ord(t_b)$.
Then 
$$w(a,b)=\sigma\tau_q(a,b)=(\sigma(a),\sigma(b))+k\delta$$ 
$$w^{-1}(\sigma(a),\sigma(b))=\tau_q\sigma^{-1}(\sigma(a),\sigma(b))=\tau_q(a,b)=(a,b)+k\delta$$
It follows that $w(a,b)$ and $w^{-1}(\sigma(a),\sigma(b))$ are positive if $k>0$ and negative if $k<0$.
Accordingly, we have $w<ws_r$, $w<s_lw$ if $\ord(t_a)>\ord(t_b)$, and $w>ws_r, w>s_lw$ if $\ord(t_a)<\ord(t_b)$.
Similarly, it follows from
$$s_lw(a,b)=s_l\sigma\tau_q(a,b)=s_l(\sigma(a),\sigma(b))+k\delta=(\sigma(b),\sigma(a))+k\delta$$ 
that $s_lw<s_lws_r$ if $\ord(t_a)>\ord(t_b)$ and $s_lw>s_lws_r$ if $\ord(t_a)<\ord(t_b)$.
Finally,
$$(ws_r)^{-1}(\sigma(a),\sigma(b))=s_r\tau_q^{-1}\sigma^{-1}(\sigma(a),\sigma(b))=s_r\tau_q^{-1}(a,b)=s_r(a,b)-k\delta=(b,a)-k\delta$$
implies $s_lws_r<ws_r$ if $\ord(t_a)<\ord(t_b)$ and $s_lws_r>ws_r$ if $\ord(t_a)>\ord(t_b)$.
\end{proof}

\section{Lusztig's Embedding and the Springer Resolution}
In this section, we define the embedding $\phi_P:T^*G/P\hookrightarrow\gL/\Para$ and relate it to Lusztig's embedding $\psi:\mathcal N\rightarrow\gL/\gOhat$ via the Springer resolution.
We also discuss Mirkovi\'c and Vybornov's (\cite{mv}) compactification of $T^*G/P$ and show how it relates to the map $\phi_P$.
\\
\\
Let $V$ be the vector space $\mathbf k^n$ with standard basis $\left\{\e_i\middle\vert\ 1\leq i\leq n\right\}$, and let the group $\gl=SL_n(\mathbf k)$ act on $V$ in the usual way. 
Fix a sequence $0=\Lambda_0<\Lambda_1<\ldots<\Lambda_{r-1}<\Lambda_r=n$, and let $P\supset B$ be the parabolic subgroup corresponding to the set of roots $\Simple_\para=\simple\backslash\left\{\alpha_{\Lambda_i}\mid 1\leq i<r\right\}$ in \simple.
Let $\lambda_i=\Lambda_i-\Lambda_{i-1}$ and $\boldsymbol\lambda=(\lambda_1,\ldots,\lambda_r)$.
We denote by \Para\ the parabolic subgroup $\Borel\subset\Para\subset\gLhat$ corresponding to $\Simple_\para\subset\Simple$. 

\subsection{Partitions}
A partition $\boldsymbol\mu$ of $n$ is a decreasing sequence $(\mu_1,\mu_2,\ldots,\mu_r)$ of positive integers such that $\sum\mu_i=n$.
Given $\boldsymbol\mu=(\mu_1,\ldots,\mu_r)$, we sometimes view $\boldsymbol\mu$ as an infinite non-negative sequence by defining $\mu_i=0$ for $i>r$.
Let $\mathcal Par$ denote the set of partitions of $n$.
The \emph{dominance order} $\preceq$ on $\mathcal Par$ is given by \begin{align*}
\hspace{80pt}\boldsymbol\mu\preceq\boldsymbol\nu&\iff\sum\limits_{j\leq i}\mu_j\leq\sum\limits_{j\leq i}\nu_j&\forall i\in\mathbb Z
\end{align*}

\subsection{The Nilpotent Cone}
Let \Ni\ be the variety of nilpotent $n\times n$ matrices, and let \gl\ act on \Ni\ by conjugation.
The variety \Ni\ has the \gl-orbit decomposition $\Ni=\bigsqcup\limits_{\boldsymbol\nu\in\mathcal Par}\Ni^\circ_{\boldsymbol\nu}$, where $\Ni_{\boldsymbol\nu}^\circ$ is the set of nilpotent matrices of Jordan type $\boldsymbol\nu$.
We write $\Ni_{\boldsymbol\nu}$ for the closure (in \Ni) of $\Ni_{\boldsymbol\nu}^\circ$.
Then $\Ni_{\boldsymbol\mu}\subset\Ni_{\boldsymbol\nu}$ if and only if $\boldsymbol\mu\preceq\boldsymbol\nu$.

\subsection{Lusztig's Embedding}
Let $\boldsymbol\nu=(\nu_1,\ldots,\nu_s)$ be the partition of $n$ which is conjugate to the partition of $n$ obtained from $\boldsymbol\lambda$ by rearranging the $\lambda_i$ in decreasing order.
In particular, $r=\nu_1$ and $s=\max\left\{\lambda_i\mid 1\leq i\leq r\right\}$.
Consider the element $\tau_\betac\in\What$ given by the affine permutation matrix\begin{align}
\label{betac}
    \tau_\betac=\sum\limits_{i=1}^st^{\nu_i-1}\E_{i,i}+\sum\limits_{i=s+1}^n t^{-1}\E_{i,i}
\end{align} 

There exists a \gl-equivariant injective map $\psi:\Ni_{\boldsymbol\nu}\hookrightarrow X_{\gOhat}(\tau_\betac)$ given by\begin{align}
\label{defn:psi}
\hspace{47pt}    \psi\left(X\right)&=\left(1-t^{-1}X\right)(\mod\gOhat)  &\text{for }X\in\Ni_{\boldsymbol\nu}
\end{align}
The map ${\psi}{\vert}{\Ni_{\boldsymbol\nu}}$ is an open immersion onto the \emph{opposite Bruhat cell} 
$$Y_{\gOhat}(\tau_q):=\Borel^-/\gO\bigcap X_{\gOhat}(\tau_q)$$ 
where $\Borel^-/\gO$ denotes the image of $\Borel^-$ under the map $\gL\rightarrow\gL/\gO$.
This statement is well-known.
A proof of can be found in \S4.1 of \cite{ah}.
A variant of the map $\psi$ was first introduced in \cite{gl:green}.
\ifminuscule
Observe that $\psi$ is \gl-equivariant:\begin{align*}
    \psi\left(gNg^{-1}\right)  &=\left(1-t^{-1}gNg^{-1}\right)(\mod\gOhat)\\
                    &=g\left(1-t^{-1}N\right)g^{-1}(\mod\gOhat)\\
                    &=g\left(1-t^{-1}N\right)(\mod\gOhat)
\end{align*}

\begin{lemma}
[cf. \cite{gl:green}]
The map $\psi$ is injective.
\end{lemma}
\begin{proof}
Suppose $\psi\left(N\right)=\psi\left(N_1\right)$, i.e., $\left(1-t^{-1}N\right)=\left(1-t^{-1}N_1\right)(\mod\gOhat)$. 
Then $$\left(1-t^{-1}N\right)^{-1}\left(1-t^{-1}N_1\right)=\left(1+t^{-1}N+t^{-2}N^2+\cdots\right)\left(1-t^{-1}N_1\right)\in\gOhat$$
In particular, $\left(1-t^{-1}N\right)^{-1}\left(1-t^{-1}N_1\right)$ is integral, and so must equal $1$. 
It follows that $N=N_1$.
\end{proof}
\fi

\subsection{The Cotangent Bundle}
\label{defCotan}
We identify the points of the variety $G/P$ with the partial flags in $V(=\mathbf k^n)$ of shape $\boldsymbol\lambda$:
$$G/P\cong\left\{(F_0\subset F_1\subset\ldots\subset F_{r-1}\subset F_r)\mid\dim F_i/F_{i-1}=\lambda_i\right\}$$
Using the Killing form on $\mathfrak{sl}_n$, we can identify the cotangent space at identity $T^*_eG/P$ with the Lie algebra \u\ of the unipotent radical $U_P$ of \para.
Then $T^*\gl/\para$ is the fiber bundle over $\gl/\para$ associated to the principal \para-bundle $\gl\rightarrow\gl/\para$, for the adjoint action of \para\ on \u:\begin{align*}
    T^*\gl/\para=\gl\times^\para\u=\gl\times\u/\sim
\end{align*}
The equivalence relation $\sim$ is given by $(g,\Y)\sim(gp,p^{-1}\Y p)$, where $g\in\gl$, $\Y\in\u$, $p\in\para$.
\def\cot{\ensuremath{\widetilde\Ni_{\boldsymbol\lambda}}}
We also have the identification\begin{align}\label{defcotan2}
    T^*G/P  &=\left\{(X,F_0\subset\ldots\subset F_k)\mid X\in\Ni,\,\dim F_i/F_{i-1}=\lambda_i,\,X(F_i)\subset F_{i-1}\right\}
\end{align}
The two identifications are related via the isomorphism 
$$(g,X)\mapsto (gXg^{-1},gV_0\subset\ldots\subset gV_k)$$
where $(V_0\subset\ldots\subset V_r)$ is the flag of shape $\boldsymbol\lambda$ fixed by $P$.

\subsection{The Affine Grassmannian}
\label{vdim}
A \emph{lattice} in $L\subset V[t,t^{-1}]$ is a $\mathbf k[t]$ module satisfying $\mathbf k[t,t^{-1}]\tensor\limits_{\mathbf k[t]} L=V[t,t^{-1}]$.
The \emph{virtual dimension} of $L$ is defined as \vskip -4pt
$$\operatorname{vdim}(L):=\dim_{\mathbf k}(L/L\cap E)-\dim_{\mathbf k}(E/L\cap E)$$
where $E$ is the standard lattice, namely the $\mathbf k[t]$ span of $V$. 
The quotient $\gLhat/\gOhat$ is an ind-variety whose points are identified with the lattices of virtual dimension $0$ (cf. \cite{gf}).
$$\gLhat/\gOhat=\left\{L\text{ a lattice}\mid \operatorname{vdim}(L)=0\right\}$$ 

\subsection{Affine Flag Varieties}
A \emph{partial affine flag} of shape $\boldsymbol\lambda$ is a sequence of lattices $L_0\subset L_1\subset\ldots\subset L_r$ satisfying $tL_r=L_0$ and $\dim L_i/L_{i-1}=\lambda_i$ for $1\leq i\leq r$.
For $P$ and \Para\ as above, we can identify the points of the ind-variety $\gLhat/\Para$ with partial affine flags of shape $\boldsymbol\lambda$ satisfying the additional condition $\operatorname{vdim}(L_0)=0$ (cf. \cite{gf}).
\begin{align}\label{plfv}
\gLhat/\Para\cong\left\{(L_0\subset\ldots\subset L_r)\mid \dim L_i/L_{i-1}=\lambda_i,\,tL_r=L_0,\,\operatorname{vdim}(L_0)=0\right\}
\end{align}
The ind-variety $\gLhat/\Para$ is called the \emph{affine flag variety} associated to $\Para$.

\subsection{The map $\phi_\para$:}
\label{defnPhi}
Let $\phi_\para:\gl\times^\para\u\rightarrow\gLhat/\Para$ be defined by \begin{align*}
\hspace{40pt}    \phi_\para(g,X)&=g(1-t^{-1}X)(\mod\Para)  &g\in\gl,\ X\in\u
\end{align*}
For $g\in\gl,\,p\in\para$ and $X\in\u$, we have \begin{align*}
    \phi_\para\left(gp,p^{-1}X p\right) &=gp\left(1-t^{-1}p^{-1}Xp\right)(\mod\Para)   \\
                                        &=g\left(p-t^{-1}Xp\right)(\mod\Para)          \\
                                        &=g\left(1-t^{-1}X\right)(\mod\Para)            \\
                                        &=g\,\phi_\para\left(1,X\right)
\end{align*}
It follows that $\phi_\para$ is well-defined and \gl-equivariant.
Under the identifications of \Cref{defcotan2,plfv}, we have\begin{align*}
    \phi_P(X,F_0\subset F_1\subset\ldots\subset F_r)&=L_0\subset L_1\subset\ldots\subset L_r
\end{align*}
where $L_i=\left(1-t^{-1}X\right)\left(V[t]\oplus t^{-1}F_i\right)$.

\begin{lemma}
The map $\phi_\para$ is injective.
\end{lemma}
\begin{proof}
Suppose $\phi_\para(g,\Y)=\phi_\para(g_1,\Y_1)$, i.e., \begin{align*}
                &g\left(1-t^{-1}\Y\right)=g_1\left(1-t^{-1}\Y_1\right)(\mod\Para) \\
    \implies    &g\left(1-t^{-1}\Y\right)=g_1\left(1-t^{-1}\Y_1\right)x
\end{align*}
for some $x\in\Para$. 
Denoting $h=g_1^{-1}g$ and $\Y'=h\Y h^{-1}$, we have
 \begin{align*}
    h\left(1-t^{-1}\Y\right)&=\left(1-t^{-1}\Y_1\right)x\\
    \implies    x           &=\left(1-t^{-1}\Y_1\right)^{-1}h\left(1-t^{-1}\Y\right)\\
                            &=\left(1-t^{-1}\Y_1\right)^{-1}\left(1-t^{-1}\Y'\right)h\\
    \implies    xh^{-1}     &=\left(1+t^{-1}\Y_1+t^{-2}\Y_2+\cdots\right)\left(1-t^{-1}\Y'\right)
\end{align*}
Now since $x\in\Para,\,h\in\gl$, the left hand side is integral, i.e., does not involve negative powers of $t$.
Hence both sides must equal identity. 
It follows $x=h\in\para=\Para\bigcap\gl$ and
$                \Y_1=\Y'=h\Y h^{-1}$.
In particular, $\left(g_1,\Y_1\right)=\left(gh^{-1},h\Y h^{-1}\right)\sim\left(g,\Y\right)$ as required.
\end{proof}

\subsection{The Springer Resolution}
\label{springer}
Let $\boldsymbol\nu$ be the partition of $n$ which is conjugate to the partition of $n$ obtained from $\boldsymbol\lambda$ by rearranging the $\lambda_i$ in decreasing order.
The Springer map $\theta:T^*\gl/P\rightarrow\Ni$ given by $\theta\left(g,X\right)=gXg^{-1}$ where $g\in\gl,\, X\in\mathfrak n$ is a resolution of singularities for the \gl-orbit $\Ni_{\boldsymbol\nu}\subset\Ni$.
The maps $\phi_P$ and $\psi$ from \Cref{defnPhi,defn:psi} sit in the following commutative diagram:
\begin{center}
\includegraphics{diag2.pdf}
\end{center}
where $\operatorname{pr}:\gLhat/\Para\rightarrow\gLhat/\gOhat$ is the natural projection. 
We present a more precise version of this statement in \Cref{springer2}.

\subsection{The Mirkovi\'c-Vybornov Compactification}
\label{mv}
Consider the \emph{convolution Grassmannian} $\widetilde{\mathcal Gr}{}^{\boldsymbol\lambda}$ whose points are identified with certain lattice flags:
$$\widetilde{\mathcal Gr}{}^{\boldsymbol\lambda}=\left\{L_0\subset L_1\subset\ldots\subset L_r\mid L_r=t^{-1}V[t],\,\dim L_i/L_{i-1}=\lambda_i,\,tL_i\subset L_{i-1}\right\} $$
In \cite{mv}, Mirkovi\'c and Vybornov construct an embedding $\widetilde\psi:T^*G/P\hookrightarrow\widetilde{\mathcal Gr}{}^{\boldsymbol\lambda}$ given by \begin{align*}
    \widetilde\psi(X,F_0\subset\ldots\subset F_r)   &=L_0\subset\ldots\subset L_r
\end{align*}
where $L_i=(1-t^{-1}X)V[t]\oplus t^{-1}F_i$.
Once again, we have a commutative diagram\begin{center}
\includegraphics{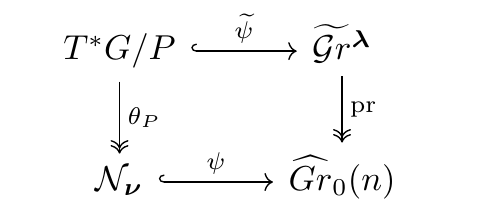}
\end{center}
where the map $\operatorname{pr}:\widetilde{\mathcal Gr}{}^{\boldsymbol\lambda}\rightarrow\mathcal Gr$ is given by $(L_0,L_1,\ldots,L_r)\mapsto L_0$.
As we can see, the incidence relations $tL_i\subset L_{i-1}$ are in general different from the incidence relations $tL_r=L_0$ of the partial affine flag variety; accordingly, $\widetilde{\mathcal Gr}{}^{\boldsymbol\lambda}$ is different from $\gLhat/\Para$.
However, when $P$ is maximal, i.e., $\boldsymbol\lambda=(d,n-d)$ for some $d$, we have an isomorphism $\beta:\widetilde{\mathcal Gr}{}^\lambda\xrightarrow\sim\gLhat/\Para$ given by 
$$(L_0\subset L_1\subset t^{-1}V[t])\mapsto (L_0\subset L_1\subset L_2)$$
where $L_2=t^{-1}L_0$.
In this case, one can verify that $\beta\circ\widetilde\psi=\phi_P$.

\section{The element $\kappa$}
\label{tableau}
Let $\boldsymbol\lambda=(\lambda_1,\ldots,\lambda_r)$, $P$, and \Para\ be as in the previous section.
Further, let $\boldsymbol\nu=(\nu_1,\ldots,\nu_s)$ be the partition of $n$ which is conjugate to the partition of $n$ obtained from $\boldsymbol\lambda$ by rearranging the $\lambda_i$ in decreasing order.
\\
\\
In this section, we describe the element $\kappa\in\What^\Para$ for which $\phi_P(T^*G/P)\subset X_\Para(\kappa)$; further, $\kappa$ is minimal for this property.
We compute $\dim X_\Para(\kappa)=l(\kappa)$ and show that when \para\ is maximal, $X_\Para(\kappa)$ is a compactification of $T^*G/P$.
\ifminuscule
\begin{figure}[h]
\centering
\begin{center}
\begin{tikzpicture}[scale=.35]\footnotesize
 \pgfmathsetmacro{\xone}{0}
 \pgfmathsetmacro{\xtwo}{ 16}
 \pgfmathsetmacro{\yone}{0}
 \pgfmathsetmacro{\ytwo}{16}

\begin{scope}<+->;
 grid
  \draw[step=1cm,gray,very thin] (\xone,\yone) grid (\xtwo,\ytwo);

\end{scope}

\begin{scope}[thick,red]
  \foreach \x in {5,8,14,15,16}
    \draw (0, 16) rectangle (5,11);
    \draw (5, 11) rectangle (8,8);
    \draw (8, 8) rectangle (10,6);
    \draw (10, 6) rectangle (11,5);
    \draw (11, 5) rectangle (14,2);
    \draw (14, 2) rectangle (16,0);
 
\draw[black] (2.5,13.5) node {\Large$\lambda_1$};
\draw[black] (6.5,9.5) node {\Large$\lambda_2$};
\draw[black] (9,7) node {\large$\bullet$};
\draw[black] (10.5,5.5) node {$\bullet$};
\draw[black] (12.5,3.5) node {\large$\bullet$};
\draw[black] (15,1) node {\large$\lambda_r$};
\end{scope}
\end{tikzpicture}
\end{center}
\caption{Elements of $\para$}
\label{diagramtofu}
\end{figure}
The elements of \para\ are exactly those matrices in \gl\ whose bottom left entries in \Cref{diagramtofu} are $0$.
Let $U$ be the unipotent radical of \para, and \u\ the Lie algebra of $U$. 
A matrix $X$ is in \u\ if and only if $X(V_i)\subset X(V_{i-1})$.
The algebra \u\ is nilpotent and contains exactly those matrices whose non-zero entries are confined to the top right corner of \Cref{diagramtofu}.
It is clear from the figure that \begin{align*}
\dim\gl/\para=\dim\u=\dfrac{n^2-\sum\lambda_i^2}{2}=\sum\limits_{i<j}\lambda_i\lambda_j
\end{align*}
\def\gln{\relax}%
\fi

\subsection{Tableaux}
\label{41}
We draw a left-aligned tableau with $r$ rows, with the $i^{th}$ row from top having $\lambda_i$ boxes. 
Fill the boxes of the tableau as follows: the entries of the $i^{th}$ row are the integers $k$ satisfying $d_i<k\leq d_{i+1}$, written in increasing order.
We denote by $\row(i)$ the set of entries in the $i^{th}$ row of the tableau.
Observe that the number of boxes in the $i^{th}$ column from the left is $\nu_i$. 
The Weyl group $W_\para$ is the set of elements in $S_n$ that preserve the partition $\left\{1,\ldots,n\right\}=\bigsqcup\limits_i\row(i)$.
\\
\\
We define a co-ordinate system $\fij\bullet.\bullet.$  on $\left\{1,\ldots,n\right\}$ as follows:
For $1\leq i\leq r$, $1\leq j\leq\nu_i$, let $\fij i.j.$ denote the $j^{th}$ entry (from the top) of the $i^{th}$ column. 
Note that $\fij i.j.$ need not be in $\row(j)$.
\\
\\ 
Finally, let $\f j,k.$ denote the elementary matrix $\E_{\fij i.j.,\fij i.k.}$.
Observe that $\fij i.b.=\fij j.c.$ if and only if $i=j$ and $b=c$.
In particular, \begin{align}\label{delta}
    \f a,b.F^j_{c,d}=\delta_{ij}\delta_{bc}\f a,d.
\end{align}

\subsection{\red\ and \blue}
\label{redb4blue}
We split the set $\left\{1,\ldots,n\right\}$ into disjoint subsets \red\ and \blue, depending on their positions in the tableau. 
Let $\set_1=\left\{\fij i.1.\mid 1\leq i\leq s\right\}$ be the set of entries which are topmost in their column.
We write $\set_1(i)=\set_1\bigcap\row(i)$.
For convenience, we also define $\set_2(i)=\row(i)\backslash\set_1(i)$ and $\set_2=\bigcup\limits_i\set_2(i)$.
\\
\\
The set $\red(i)$ is the collection of the $\#\set_1(i)$ smallest entries in $\row(i)$: 
$$\red(i):=\left\{j\mid\Lambda_{i-1}<j\leq\Lambda_i-\max\left\{\lambda_j\mid j<i\right\}\right\}$$
We set $\blue(i)=\row(i)\backslash\red(i)$, $\red=\bigcup\limits_i\red(i)$, and $\blue=\bigcup\limits_i\blue(i)$.
The elements of $\red(i)$ are smaller than the elements of $\blue(i)$.
\\
\\
The elements of \red, arranged in increasing order are written $l(1),\ldots,l(s)$.
We enumerate the elements of $\blue$, written row by row from bottom to top, each row written left to right, as $m(1),\ldots,m(n-s)$.

\ifdraft\marginpar{example.tex}\fi
    \def\ten{10}
    \def\eleven{11}
    \def\eleven{11}
    \def\twelve{12}
    \def\thirteen{13}
    \def\fourteen{14}
    \def\fifteen{15}
    \def\sixteen{16}
    \def\seventeen{17}

\begin{example} Let $n=17$ and the sequence $(\Lambda_i)$ be $(1,5,9,11,17)$. 
The corresponding tableau is
\begin{align*}
    \young(1,2345,6789,\ten\eleven,\twelve\thirteen\fourteen\fifteen\sixteen\seventeen)
\end{align*}
\begin{itemize}
    \item $r=5$, $s=6$.
    \item The sequence $(\lambda_i)$ is $(1,4,4,2,6)$.
    \item The sequence $(\nu_i)$ is $(5,4,3,3,1,1)$.
    \item $\row(3)=\left\{6,7,8,9\right\}$.
    \item $\fij1.4.=10,\ \fij4.3.=15,\ \fij6.1.=17$ etc.
    \item $F^1_{2,4}=\E_{2,10},\ F^3_{3,3}=\E_{14,14}$ etc. 
    \item $\set_1=\left\{1,3,4,5,16,17\right\}$.
    \item The sequence $l(i)$ is $(1,2,3,4,12,13)$.
    \item The sequence $m(i)$ is $(14,15,16,17,10,11,6,7,8,9,5)$.
\end{itemize}
\end{example}
\ifdraft\marginpar{example.tex}\fi

\ifdraft\marginpar{appendix.tex}\fi

\ifminuscule
\def\e{\ensuremath{e}}
\begin{prop}
Let $\boldsymbol{\nu}=(\nu_1,\ldots,\nu_r)$ and $\boldsymbol{\nu'}=(\nu'_1,\ldots,\nu'_s)$ be conjugate partitions, written in non-increasing order. 
Then \begin{align*}
    \sum\limits_{i=1}^r\nu_i^2=\sum\limits_{i=1}^s\sum\limits_{j=1}^s\min\{\nu'_i,\nu'_j\}=\sum\limits_{i=1}^s(2i-1)\nu'_i.
\end{align*}
\end{prop}
\begin{proof}
Let $\left\langle\bullet,\bullet\right\rangle$ be the dot product on $V$ given by $\langle \e_i,\e_j\rangle=\delta_{ij}$. 
Consider the Young diagram $\mathrm{\mathbf Y}$ whose $i^{th}$ row has $\nu_i$ boxes. 
Fill each box in the $i^{th}$ row of $\mathrm{\mathbf Y}$ with $\e_i$.
Let $v_i$ be the sum of all vectors in the $i^{th}$ coloumn and $\mathbf{v}$ the sum of the vectors in all the boxes in $\mathrm{\mathbf Y}$. 
Observe that $
    v_i =\sum\limits_{1\leq i\leq\nu'_i}\e_i$
and so\begin{align*}
    \langle v_i, v_j\rangle &=\left\langle\sum\limits_{k=1}^{\nu_i'}\e_k,\sum\limits_{k=1}^{\nu_j'}\e_k\right\rangle
                            =\min\{\nu'_i,\nu'_j\}\\
    \langle\mathbf v,\mathbf v\rangle   &=\sum\limits_{i=1}^s\sum\limits_{j=1}^s\langle v_i,v_j\rangle=\sum\limits_{i=1}^s\sum\limits_{j=1}^s\min\{\nu'_i,\nu'_j\}
\end{align*}
But we also have \begin{align*}
    \langle\mathbf{v},\mathbf v\rangle  &=\left\langle\sum\nu_i\e_i,\sum\nu_i\e_i\right\rangle
                                        =\sum\limits_{i,j}\delta_{ij}\nu_i\nu_j=\sum\nu_i^2 
\end{align*}
which gives the first part of the equality.
For the second part, we use $\min\{\nu'_i,\nu'_j\}=\nu'_{\max\{i,j\}}$ to get \begin{align*}
    \left\langle\mathbf v,\mathbf v\right\rangle    &=\left\langle\sum v_i,\sum\ v_j\right\rangle
                                                    =\sum\left\langle v_i, v_i\right\rangle+2\sum\limits_{i>j}\left\langle v_i, v_j\right\rangle\\
                                                    &=\sum\nu'_i+2\sum\limits_{i=1}^s\sum\limits_{j=1}^{i-1}\nu'_i 
                                                    =\sum\nu'_i+2\sum(i-1)\nu'_i
                                                    =\sum(2i-1)\nu'_i
\end{align*}
\end{proof}
\ifdraft\marginpar{appendix.tex}\fi

\begin{cor}
\label{magic}
Let $\boldsymbol\lambda, \boldsymbol\nu$ be as in \Cref{41}. Then:\begin{align*}
    \sum\limits_{i=1}^r\lambda_i^2=\sum\limits_{i=1}^s(2i-1)\nu_i.
\end{align*}\end{cor}
\begin{proof}
The collection $(\nu_1,\ldots,\nu_s)$ is a partition of $n$, and the collection $(\lambda_1,\ldots,\lambda_r)$ is a permutation of the conjugate partition $(\nu'_1,\ldots,\nu'_r)$.
\end{proof}
\fi

\subsection{The element $\kappa$}
\label{formula:kappa}
\label{springer2}
Let $\kappa\in\What$ be given by the affine permutation matrix $$\sum\limits_{i=1}^st^{\nu_i-1}\E_{i,l(i)}+\sum\limits_{i=1}^{n-s}t^{-1}\E_{i+s,m(i)}$$
The commutative diagram of \Cref{springer} can be refined to the following:
\begin{center}
\includegraphics{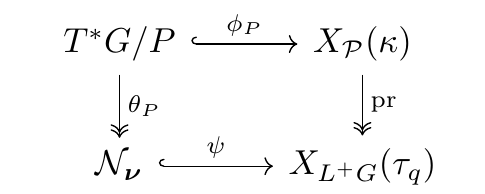}
\end{center}
The only additional part is the claim $\phi_P(T^*G/P)\subset X_\Para(\kappa)$. 
This is the content of \Cref{injective}.
\ifminuscule
Let $\phi_\para$ be as in \Cref{defnPhi}.
In this section, we show that the image of $\phi_\para$ is contained in $X_\Para(\kappa)$. 
We first construct a matrix \Z\ whose \para-orbit is dense in \u\ and compute the Bruhat cell \w\ containing $(1-t^{-1}Z)$.
We discuss the relationship between \w\ and $\kappa$, and then prove $\image(\phi_\para)\subset X_\Para(\kappa)$.
\fi
\subsection{The Matrix \Z}
\label{defn:z}
Let \Z\ be the $n\times n$ matrix given by 
$\Z=\sum\limits_{i=1}^{s}\sum\limits_{j=1}^{\nu_i-1}\f j,j+1.$.
Then \vskip -0.25em
$$Z\e_{\fij i.j.}=\begin{cases}\e_{\fij i.j-1.}&\text{if }\ j>1\\ 0&\text{if }\ j=1\end{cases}$$
It follows that $\left\{\e_{\fij i.\nu_i.}\middle\vert\ 1\leq i\leq s\right\}$ is a minimal generating set for $V(=\mathbf k^n)$ as a module over $\mathbf k[\Z]$ (the $\mathbf k$-algebra generated by \Z).
In particular, 
$\Z(V_i)\subset V_{i-1}$ for all $i$.
It follows that $\Z\in\u$, where \u\ is as in \Cref{defCotan}.
\\
\\
An element in $x\in\u$ is called a {\em Richardson element} of \u\ if the \para-orbit of $x$ is dense in \u, or equivalently, the \gl-orbit of $(1,\Z)$ is dense in $T^*\gl/\para=\gl\times^\para\u$.
A comprehensive study of Richardson elements in \u\ can be found in \cite{he}.
In particular, we will need the following:

\begin{lemma}
[Theorem 3.3 of \cite{he}]
\label{conjugacy}
\label{dense}
The matrix \Z\ is a Richardson element in \u.
\end{lemma}
\ifminuscule
Let $K$ be the group of non-zero scalar matrices over $\mathbf k$, and $\widetilde\para$ the subgroup of  $\operatorname{GL}_n(\mathbf k)$ generated by $K$ and \para.
Since $K$ acts trivially on \u, it is enough to show that the orbit of \Z\ under $\widetilde\para$ is dense in \u.
Let $C_{\widetilde\para}(\Z)$ (resp. $C(\Z)$) be the stabilizer of \Z\ under the conjugation action of $\widetilde\para$ (resp. $\operatorname{GL}_n(\mathbf k)$).
It is enough to show that $\dim C_{\widetilde\para}(Z)=\dim\widetilde\para-\dim\u$.
Observe first that \begin{align*}
    \dim C(\Z)\geq\dim C_{\widetilde\para}(\Z)\geq\dim\widetilde\para-\dim\u=\sum\limits_{i=1}^r\lambda_i^2.
\end{align*}
It is now enough to show that $\dim C(\Z)\leq\sum\lambda_i^2$, from which the result will follow.
\\
\\
Let $\zeta:\operatorname{GL}_n(\mathbf k)\hookrightarrow V^{\oplus n}$ be the map $g\mapsto(ge_1,\ldots,ge_n)$, and $\rho:V^{\oplus n}\rightarrow V^{\oplus s}$ the projection $(v_1,\ldots,v_n)\mapsto(v_{\fij 1.\nu_1.},\ldots,v_{\fij s.\nu_s.})$.
As a $\mathbf k[\Z]$-module, $V$ is generated by the set $\left\{\e_{\fij i.\nu_i.}\middle\vert\ 1\leq i\leq s\right\}$. 
Since $V$ is generated (as a $\mathbf k[\Z]$-module) by the set $\left\{\e_{\fij i.\nu_i.}\middle\vert\ 1\leq i\leq s\right\}$,
the action of $g\in C(\Z)$ on $V$ is determined by its action on the set $\left\{\e_{\fij i.\nu_i.}\middle\vert\ 1\leq i\leq s\right\}$. 
In particular, the map $\left(\rho\circ\zeta\right)\vert_{C(\Z)}$ is injective.
Furthermore, since $\Z^{\nu_i}\e_{\fij i.\nu_i.}=0$, it follows that $\image(p\circ\zeta)\vert_{C(\Z)}\subset\bigoplus\limits_{i=1}^s\ker\Z^{\nu_i}$.
The kernel of $\Z^{\nu_i}$ being spanned by $\left\{e_{\fij j.k.}\middle\vert\ j\leq s,\ k\leq\nu_i,\ k\leq\nu_j\right\}$, we deduce \begin{align*}
    \dim C(\Z)          \leq&\dim\left(\bigoplus_{i=1}^s\ker\Z^{\nu_i}\right)
                        =\sum\limits_{i=1}^s\dim\left(\ker\Z^{\nu_i}\right)\\
                        &=\sum\limits_{i=1}^s\sum\limits_{j=1}^s\#\left\{\e_{\fij j.k.}\mid k\leq\min\left\{\nu_i,\nu_j\right\}\right\} \\
                        &=\sum\limits_{i=1}^s\sum\limits_{j=1}^s\min\left\{\nu_i,\nu_j\right\}
                        =\sum\limits_{i=1}^r\lambda_i^2
\end{align*}
where the last equality is from \Cref{magic}.
\fi

\ifminuscule
\begin{cor}
\label{centralizerInP}
We have the equality $C_\gl(\Z)=C_\para(\Z)$. 
\end{cor}
\begin{proof}
Since $C_\gl(\Z)=C(\Z)\bigcap\gl$ and $C_\para(\Z)=C_{\widetilde\para}(\Z)\bigcap\gl$, it is enough to show $C(\Z)=C_{\widetilde\para}(\Z)$.
The map $p\circ\zeta$ embeds $C(\Z)$ as a dense, hence irreducible, subset of $V^{\oplus s}$.
In particular the real codimension of $C(\Z)$ in $V^{\oplus s}$ is at least $2$, and so $C(\Z)$ is connected. 
It follows that $C_{\widetilde\para}(\Z)$, being a codimension $0$ closed subvariety of $C(\Z)$, must equal $C(\Z)$.
\end{proof}
\fi
\ifdraft\marginpar{z.tex}\fi
\ifdraft\marginpar{main.tex}\fi

\begin{prop}
\label{kappa}
\label{wpp}
Let $\Borel\w\Borel$ be the Bruhat cell containing $1-t^{-1}\Z$. 
A lift of \w\ to \N\ is given by \begin{align*}
    \overset\sim\w=\sum_{i=1}^s\left(t^{\nu_i-1}\f\nu_i,1.-\sum_{j=2}^{\nu_i}t^{-1}\f j-1,j.\right)
\end{align*}
\end{prop}
\begin{proof}
For $1\leq i\leq s$, let \begin{align*}
    b_i         &:=\sum\limits_{j=1}^{\nu_i}\sum\limits_{k=j}^{\nu_i}t^{k-j}\f k,j. 
                =\sum\limits_{j=1}^{\nu_i}\sum\limits_{k=0}^{\nu_i-j}t^k\f j+k,j. \\
    c_i         &:=\sum\limits_{j=1}^{\nu_i}\f j,j.+\sum\limits_{j=2}^{\nu_i}t^{j-1}\f j-1,1. \\
    \Z_i        &:=\sum\limits_{j=1}^{\nu_i}\f j,j.-t^{-1}\sum\limits_{j=1}^{\nu_i-1}\f j,j+1.
\end{align*}
We compute\begin{align*}
    b_i\Z_ic_i  &=\left(\sum\limits_{j=1}^{\nu_i}\sum\limits_{k=0}^{\nu_i-j}t^k\f j+k,j.\right)\left(\sum\limits_{j=1}^{\nu_i}\f j,j.-t^{-1}\sum\limits_{j=1}^{\nu_i-1}\f j,j+1.\right)c_i\\
                &=\left(\sum\limits_{j=1}^{\nu_i}\sum\limits_{k=0}^{\nu_i-j}t^k\f j+k,j.-\sum_{j=1}^{\nu_i-1}\sum_{k=0}^{\nu_i-j}t^{k-1}\f j+k,j+1.\right)c_i\\
                &=\left(\sum\limits_{j=1}^{\nu_i}\sum\limits_{k=0}^{\nu_i-j}t^k\f j+k,j.-\sum_{j=2}^{\nu_i}\sum_{k=-1}^{\nu_i-j}t^k\f j+k,j.\right)c_i\\
                &=\left(\sum_{k=0}^{\nu_i-1}t^k\f 1+k,1.-\sum_{j=2}^{\nu_i}t^{-1}\f j-1,j.\right)\left(\sum\limits_{j=1}^{\nu_i}\f j,j.+\sum\limits_{j=2}^{\nu_i}t^{j-1}\f j-1,1.\right)\\
                &=\sum_{k=0}^{\nu_i-1}t^k\f k+1,1.-\sum_{j=2}^{\nu_i}t^{-1}\f j-1,j.-\sum\limits_{j=2}^{\nu_i}t^{j-2}\f j-1,1.\\
                &=t^{\nu_i-1}\f\nu_i,1.-\sum_{j=2}^{\nu_i}t^{-1}\f j-1,j.
\end{align*}
Observe that $1-t^{-1}\Z=\sum\limits_{1\leq i\leq s}\Z_i$.
It follows from \Cref{delta} that for $i\neq j$, $b_i\Z_j=0$ and $\Z_jc_i=0$.
Writing $b=\sum\limits_ib_i$ and $c=\sum\limits_ic_i$, we see \begin{align*}b(1-t^{-1}\Z)c=\sum\limits_i b_i\Z_ic_i=\w\end{align*}
The result now follows from the observation $b,c\in\Borel$.
\end{proof}

\begin{lemma}
\label{samecoset}
There exist $w_g\in\What_{\gOhat}(=W),\ w_p\in\What_\Para(=W_P)$ such that $\w=w_g\kappa w_p$.
In particular, $X_\Para(\w)\subset\overline{\gOhat\kappa\Para}(\mod\Para)$.
\end{lemma}
\begin{proof}
Recall the disjoint subsets $\set_1, \set_2$ of $\left\{1,\ldots,n\right\}$ from \Cref{redb4blue}.
Consider the bijection $\iota:\set_2\rightarrow\left\{\fij i.j.\mid 1\leq j\leq \nu_i-1\right\}$ given by $\iota(\fij i.j.)=\fij i.j-1.$. 
We reformulate \Cref{wpp} as 
$$    \overset\sim\w=\sum\limits_{i=1}^st^{\nu_i-1}\E_{\fij i.\nu_i.,\fij i.1.}-\sum\limits_{i\in\set_2}\E_{\iota(i),i}$$
It follows from \Cref{apm} that the affine permutation matrix of $\w$ is given by
$$    \w=\sum\limits_{i=1}^st^{\nu_i-1}\E_{\fij i.\nu_i.,\fij i.1.}+\sum\limits_{i\in\set_2}\E_{\iota(i),i}$$
Observe that $\#\red(k)=\#\set_1(k)$ and $\#\blue(k)=\#\set_2(k)$.
Since both $(l(i))_{1\leq i\leq s}$ and $(\fij i.1.)_{1\leq i\leq s}$ are increasing sequences, $\fij i.1.$ and $l(i)$ are in the same row for each $i$.
Furthermore, there exists an enumeration $t(1),\ldots,t(n-s)$ of $\set_2$ such that $t(i)$ is in the same row as $m(i)$ for all $i$.
We define $w_g\in W$ and $w_p\in\W_P$ via their affine permutation matrices: \begin{align*}
    w_g&=\sum\limits_{i=1}^s\E_{i,\fij i.\nu_i.}+\sum\limits_{i=1}^{n-s}\E_{i+s,\iota(t(i))}\\
    w_p&=\sum\limits_{i=1}^s\E_{\fij i.1.,l(i)}+\sum\limits_{i=1}^{n-s}\E_{t(i),m(i)}
\end{align*}
A simple calculation shows $\w=w_g\kappa w_p$.
\end{proof}

\begin{prop}
\label{gstable}
The Schubert variety $X_\Para(\kappa)$ is stable under left multiplication by \gOhat, i.e., $X_\Para(\kappa)=\overline{\gOhat\kappa\Para}(\mod\Para)$.
\end{prop}
\begin{proof}
Consider the affine permutation matrix of $\kappa$.
We will show, for $1\leq i<n$, either $s_i\kappa=\kappa(\mod\Para)$ or $s_i\kappa<\kappa$.
We split the proof into several cases, and use \Cref{table}:\begin{enumerate}
    \item $i<s$ and $\nu_i=\nu_{i+1}$ {\bf:} We deduce from $\nu_i=\nu_{i+1}$ that the entries $l(i)$ and $l(i+1)$ appear in the same row of the tableau. 
        In particular, $l(i)\in\Simple_\para$. 
        The non-zero entries of the $i^{th}$ and $(i+1)^{th}$ row are $t^{\nu_i-1}\E_{i,l(i)}$ and $t^{\nu_{i+1}-1}\E_{i+1,l(i+1)}$ respectively.
        We see $s_i\kappa=\kappa s_{l(i)}=\kappa(\mod\What_\Para)$. 
    \item $i<s$ and $\nu_i>\nu_{i+1}$ {\bf:} The non-zero entries of the $i^{th}$ and $(i+1)^{th}$ row are $t^{\nu_i-1}\E_{i,l(i)}$ and $t^{\nu_{i+1}-1}\E_{i+1,l(i+1)}$ respectively.
        Case $2$ of \Cref{table} applies with $a=l(i),\,b=l(i+1)$, and we have $s_i\kappa<\kappa$.
    \item $i=s$ {\bf:} The non-zero entries of the $i^{th}$ and $(i+1)^{th}$ row are $t^{\nu_s-1}\E_{s,l(s)}$ and $t^{-1}\E_{s+1,m(1)}$ respectively.
        Since $m(1)\in\blue(r)$, it follows from \Cref{redb4blue} that $l(s)<m(1)$.
        Case $2$ of \Cref{table} applied with $a=l(s),\,b=m(1)$ tells us $s_i\kappa<\kappa$.
    \item $i>s$ and $m(i-s)\in\Simple_\para$ {\bf:} The non-zero entries of the $i$ and $(i+1)^{th}$ row are $t^{-1}\E_{i,m(i-s)}$ and $t^{-1}\E_{i+1,m(i+1-s)}$ respectively.
        It follows $s_i\kappa=\kappa s_{m(i-s)}=\kappa(\mod\What_\Para)$.
    \item $i>s$ and $m(i-s)\notin\Simple_\para$ {\bf:} It follows from $m(i-s)\notin\Simple_\para$ that if $m(i-s)\in\row(j)$ then $m(i+1-s)\in\row(j-1)$.
        In particular, $m(i+1-s)<m(i-s)$.
        The non-zero entries of the $i$ and $(i+1)^{th}$ row are $t^{-1}\E_{i,m(i-s)}$ and $t^{-1}\E_{i+1,m(i+1-s)}$ respectively.
        Case $1$ of \Cref{table} applied with $a=m(i+1-s),\,b=m(i+s)$ tells us $s_i\kappa<\kappa$.
\end{enumerate}
\end{proof}

\begin{theorem}
\label{injective}
Let $\phi_\para$ be as in \Cref{defnPhi}.
Then $\image(\phi_\para)\subset X_\Para(\kappa)$.
In particular, the map $\phi_\para$ gives a compactification of $T^*\gl/\para$.
Further $\kappa\in\What^\Para$ is minimal for the property $\image(\phi_\para)\subset X_\Para(\kappa)$.
\end{theorem}
\begin{proof}
Let $\mathcal O$ denote the \gl-orbit of $(1,\Z)\in\gl\times^\para\u$.
It follows from \Cref{dense} that $T^*\gl/\para=\overline{\mathcal O}$, and from \Cref{kappa} that $\phi_\para(1,Z)\in\Borel\w\Para/\Para$.
Since $\phi_\para$ is \gl-equivariant, it follows that $\phi_\para(\mathcal O)\subset\gl\Borel\w\Para/\Para\subset\gOhat\w\Para/\Para$, and so 
$$\phi_\para\left(T^*\gl/\para\right) =\phi_\para\left(\overline{\mathcal O}\right)\subset\overline{\phi_\para\left(\mathcal O\right)}\subset\overline{\gOhat\w\Para/\Para}=X_\Para(\kappa)$$
where the last equality follows from \Cref{samecoset} and \Cref{gstable}.
Further, $\overline{\phi_\para(T^*\gl/\para)}$, being a closed subvariety of $X_\Para(\kappa)$, is compact.
\\
\\
To prove the minimality of $\kappa$, we show that there exists a $a\in\gl$ such that $\phi_\para(a,\Z)\in\Borel\kappa\Para/\Para$.
\Cref{gstable} implies that $\kappa$ is maximal in the right coset $\W\w$. 
In particular, there exists $w\in\W$ such that $\kappa=w\w$ and $l(\kappa)=l(w)+l(\w)$.
It follows that $\overline{\Borel w\Borel\w\Para}=\overline{\Borel\kappa\Para}$ (see \cite{sk} for details), and $\phi_\para(a,\Z)\in X_\Para(\kappa)$ for any $a\in\Borel w\Borel$.
\end{proof}

\ifminuscule
\begin{lemma}
\label{length}
The length of $\tau_\betac(\in\What)$ is $2\dim\gl/\para$.
\end{lemma}
\begin{proof}
Let $\langle\ ,\ \rangle$ denote the dual pairing on $\cartan^*\times\cartan$.
We use \Cref{lengthoftau} to compute $l(\tau_\betac)$: \begin{align*}
    l(\tau_\betac)  &=\sum\limits_{1\leq i<j\leq n}\left\vert\left\langle \epsilon_i-\epsilon_j,\betac\right\rangle\right\vert\\
                    &=\sum\limits_{1\leq i<j\leq n}\left\lvert\left\langle \epsilon_i-\epsilon_j,\mathrm{Id}_{n\times n}-\sum\limits\nu_k\E_{kk}\right\rangle\right\rvert\\
                    &=\sum\limits_{1\leq i<j\leq n}\left\lvert\left\langle \epsilon_i-\epsilon_j,\sum\limits\nu_k\E_{kk}\right\rangle\right\rvert\\
                    &=\sum\limits_{1\leq i<j\leq s}\left\lvert\left\langle \epsilon_i-\epsilon_j,\sum\limits\nu_k\E_{kk}\right\rangle\right\rvert+\sum\limits_{\substack{1\leq i\leq s\\ s<j\leq n}}\left\lvert\left\langle\epsilon_i-\epsilon_j,\sum\limits\nu_k\E_{kk}\right\rangle\right\rvert \\ 
                    &=\sum\limits_{1\leq i<j\leq s}\left(\nu_i-\nu_j\right)+\sum\limits_{\substack{1\leq i\leq s\\ s<j\leq n}}\nu_i 
\end{align*}\begin{align*}
                    &=\sum\limits_{k=1}^{s} (s+1-2k)\nu_k+(n-s)\sum\limits_{k=1}^s\nu_i\\
                    &=s\sum\limits_{k=1}^s\nu_k - \sum\limits_{k=1}^s(2k-1)\nu_k+(n-s)n\\
                    &=sn - \sum\limits_{k=1}^{s} (2k-1)\nu_k +(n-s)n\\
                    &=n^2-\sum\lambda_i^2\qquad\qquad\qquad(\text{cf. \Cref{magic}})\\
                    &=2\dim\gl/\para.
\end{align*}
\end{proof}
\fi
\ifminuscule
\begin{theorem}
\label{springer2}
\label{lusztig}
There exists a resolution of singularities $\theta:T^*\gl/\para\rightarrow\Ni_{\boldsymbol\nu}$, given by $(g,N)\mapsto gNg^{-1}$ for $(g,N)\in\gl\times^\para\u$.
Furthermore, we have the following commutative diagram:\begin{center}
\includegraphics{diagram3.pdf}
\end{center}
where $\operatorname{pr}$ is the restriction to $X_\Para(\kappa)$ of the natural projection $\gLhat/\Para\twoheadrightarrow\gLhat/\gOhat$.
\ifdraft\marginpar{The map $\psi_\nu$ is an open embedding.}\fi 
\end{theorem}
\fi
\ifminuscule
\begin{proof}
\label{sigma}
The existence and desingularization property of $\theta_P$ are well-known, and are implied, for example, by the main results of \cite{he} and \cite{fu}.
Recall the injective map $\psi$ from \Cref{springer}; let $\psi_{\boldsymbol\nu}=\psi|\Ni_{\boldsymbol\nu}$.
Note that $\kappa=\tau_\betac\sigma$, where the affine permutation matrix of $\sigma\in W$ is given by \begin{align}
\label{form:sigma}
    \sigma=\sum\limits_{i=1}^{s}\E_{i,l(i)}+\sum\limits_{i=1}^{n-s}\E_{i+s,m(i)}
\end{align}
In particular, $X_{\gOhat}(\tau_\betac)=X_{\gOhat}(\kappa)$, and so it follows from \Cref{injective} that $\psi(\Ni_{\boldsymbol\nu})\subset X_{\gOhat}(\tau_\betac)$.
Since $\dim\Ni_{\boldsymbol\nu}=2\dim\gl/\para=l(\tau_\betac)=\dim X_{\gOhat}(\tau_\betac)$, the result follows.
\end{proof}
\fi

\ifminuscule
\section{Dimension of $X_\Para(\kappa)$}
Lakshmibai (cf. \cite{vl}) has constructed, for \para\ a maximal parabolic subgroup in \gl, a similar embedding $\eta_\para:T^*\gl/\para\hookrightarrow X_\Para(\eta)$. 
The map $\eta_\para$ is dominant, thus giving a Schubert variety as a compactification of $T^*\gl/\para$.
Lakshmibai et al. (cf. \cite{crv}) discuss a family of maps $T^*\gl/\borel\rightarrow \gLhat/\Borel$, including the map $\phi_\borel$, for which the image of $T^*\gl/\borel$ is \emph{not} a Schubert variety.
With this in mind, we compute the dimension of $X_\Para(\kappa)$, and thus the codimension of $\image(\phi_\para)$ in $X_\Para(\kappa)$. 
We also recover Lakshmibai's result, i.e., for \para\ a maximal parabolic subgroup in \gl, the map $\phi_\para$ is dominant.
In fact, \Cref{maximalresult} says that $X_\Para(\kappa)$ is a compactification of $T^*G/P$ if and only if $P$ is a maximal parabolic subgroup.
\fi

\begin{prop}
The dimension of $X_\Para(\kappa)$ is $l(\kappa)$, the length of $\kappa$.
\end{prop}
\begin{proof}
We need to show that $\kappa\in\What^\Para$. 
For $\alpha_i\in\Simple_\para$, we show $\kappa s_i>\kappa$.
Recall the partitioning of $\left\{1,\ldots,n\right\}$ from \Cref{redb4blue} into \red\ and \blue.
Note that $\alpha_i\in\Simple_\para$, implies $i$ and $i+1$ appear in the same row of the tableau.
In particular, if $i\in\Simple_\para\bigcap\blue(j)$ then $i+1\in\blue(j)$.
\begin{enumerate}
    \item   Suppose $i\in\blue$. 
            Then $i=m(k)$ and $i+1=m(k+1)$ for some $k$. 
            The non-zero entries in the $i^{th}$ and $(i+1)^{th}$ columns are $t^{-1}\E_{k+s,i}$ and $t^{-1}\E_{k+s+1,i+1}$.
            We apply Case 1 of \Cref{table} with $a=i,\,b=i+1$.
    \item   Suppose $i\in\red$ and $i+1\in\blue$.
            The non-zero entries in the $i^{th}$ and $(i+1)^{th}$ columns are $t^{\nu_k-1}\E_{k,i}$ and $t^{-1}\E_{j,i+1}$.
            Since $k\leq s<j$, we can apply Case 2 of \Cref{table} with $a=i,\,b=i+1$, $\nu_k-1=\ord(t_a)>\ord(t_b)=-1$ and $i=\sigma(a)<\sigma(b)=i+1$ to get $\kappa<\kappa s_i$.
    \item   Suppose $i,i+1\in\red$.
            Since $i$ and $i+1$ are in the same row of the tableau, we have $i=l(k)$ and $i+1=l(k+1)$ for some $k$. 
            The non-zero entries in the $i^{th}$ and $(i+1)^{th}$ columns are $t^{\nu_k-1}\E_{k,i}$ and $t^{\nu_{k+1}-1}\E_{k+1,i+1}$.
            If $\nu_k=\nu_{k+1}$, Case 1 of \Cref{table} applies with $a=i,\,b=i+1$ to give $\kappa<\kappa s_i$.
            If $\nu_k>\nu_{k+1}$, Case 2 of \Cref{table} applies with $a=i,\,b=i+1$, $\nu_k-1=\ord(t_a)>\ord(t_b)=\nu_{k+1}-1$ to give $\kappa<\kappa s_i$.
\end{enumerate}
\end{proof}
\begin{lemma}
\label{lsigma}
\label{lengthofkappa}
The length of $\kappa$ is given by the formula \begin{align*}
    l(\kappa)   &=2\dim\gl/\para+\sum\limits_{k'<k}\#\row(k)\#\blue(k')
\end{align*}
\end{lemma}
\begin{proof}
Note that $\kappa=\tau_\betac\sigma$, where $\tau_\betac$ is given by \Cref{betac}, and $\sigma\in W$ is given by the permutation matrix \begin{align}
\label{form:sigma}
    \sigma=\sum\limits_{i=1}^{s}\E_{i,l(i)}+\sum\limits_{i=1}^{n-s}\E_{i+s,m(i)}
\end{align}
Viewing $\sigma$ as an element of $S_n$, we have $\sigma^{-1}(i)=\begin{cases}l(i)&i\leq s\\m(i-s)&i>s\end{cases}$. 
In particular, \begin{align*}
    l(\sigma)=  &\#\left\{(i,j)\mid 1\leq i<j\leq n,\,\sigma^{-1}(i)>\sigma^{-1}(j)\right\} \\
             =  &\#\left\{(i,j)\mid i<j\leq s, l(i)>l(j)\right\}\\ 
                    &\qquad+\#\left\{(i,j)\mid i\leq s<j,\,l(i)>m(j-s)\right\}\\ 
                    &\qquad\qquad+\#\left\{(i,j)\mid s<i<j,\,m(i-s)>m(j-s)\right\} 
\end{align*}
Recall that $l(i)$ is an increasing sequence, i.e., $i<j<s\implies l(i)<l(j)$, and so \begin{align*}
    l(\sigma)=  &\#\left\{(i,j)\mid i\leq s,\,j\leq n-s,\,l(i)>m(j)\right\}     \\ 
                    &\qquad+\#\left\{(i,j)\mid i<j\leq n-s,\,m(i)>m(j)\right\}  \\
             =  &\#\left\{(i,j)\mid i\in\red,\,j\in\blue,\,i>j\right\}     \\ 
                    &\qquad+\#\left\{(i,j)\mid i<j\leq n-s,\,m(i)>m(j)\right\}  \\
             =  &\sum\limits_{k'<k}\#\left\{(i,j)\mid i\in\red(k),\,j\in\blue(k')\right\} \\ 
                    &\qquad+\sum\limits_{k'<k}\#\left\{(i,j)\mid i\in\blue(k),\,j\in\blue(k')\right\} \\
             =  &\sum\limits_{k'<k}\#\left\{(i,j)\mid i\in\row(k),\,j\in\blue(k')\right\} \\ 
             =  &\sum\limits_{k'<k}\#\row(k)\#\blue(k')    
\end{align*}
Now, it follows from \Cref{actionOfq,actionOfTau} that $\tau_\betac\in\What^{\gOhat}$.
In particular, 
$$l(\tau_\betac)=\dim X_{\gOhat}(\tau_\betac)=\dim\Ni_{\boldsymbol\nu}=2\dim G/P$$
Further, 
$l(\kappa)=l(\tau_\betac)+l(\sigma)=2\dim\gl/\para+\sum\limits_{k'<k}\#\row(k)\#\blue(k')$ as claimed.
\end{proof}

\begin{cor}
\label{maximalresult}
The Schubert variety $X_\Para(\kappa)$ is a compactification of $T^*G/P$ if and only if $P$ is a maximal parabolic subgroup.
\end{cor}
\begin{proof}
The parabolic subgroup $P$ is maximal if and only if $S_\para=\simple\backslash\left\{\alpha_d\right\}$ for some $d$, equivalently, the corresponding tableau has exactly $2$ rows. 
In this case $\blue\subset\row(2)$, (recall from \Cref{redb4blue} that $\blue=\bigsqcup\limits_i\blue(i)$).
It follows that the second term in \Cref{lengthofkappa} is an empty sum, which implies $l(\kappa)=2\dim\gl/\para$.
Suppose now that the tableau has $r\geq 3$ rows.
In this case, both $\blue(2)$ and $\row(r)$ are non-empty, and so the second term in \Cref{lengthofkappa} is strictly greater than $0$.
\end{proof}

\section{Conormal variety of the Schubert Divisor}
Let $\boldsymbol\lambda$, $P$, and \Para\ be as in the previous section.
In this section, we show that for $X_P(w)$ a Schubert divisor in $G/P$, the conormal variety $N^* X_P(w)$ is an open subset of a Schubert variety in $\gLhat/\Para$.
In particular, $N^* X_P(w)$ is normal, Cohen-Macaulay, and Frobenius split.
\\
\\
We write \cartan, $\mathfrak b$, and \lgl\ for the Lie algebras of $T$, $B$, and $G$ respectively. 
For $\alpha\in\roots$, we denote by $\lgl^\alpha$ the root space corresponding to $\alpha$.

\subsection{The Conormal Variety}
Let $X$ be a closed subvariety in $G/P$, and write $X_{\operatorname{sm}}$ for the smooth locus of $X$.
For $x\in X_{\operatorname{sm}}$, the \emph{conormal fibre} $\con_x$ is the annihilator of $T_xX$ in $T^*_xG/P$.
The \emph{conormal variety} $N^* X$ of $X\hookrightarrow\gl/P$ is then defined to be the closure in $T^*\gl/P$ of the conormal bundle $N^*X_{\operatorname{sm}}$.

\begin{prop}
\label{prop:uw}
Let $\Roots_P$ be the subset of \roots\ generated by $S_P$. 
The conormal variety $N^* X_P(w)$ is the closure in $T^*\gl/P$ of 
$$  \left\{(bw,X)\in\gl\times^P\u\mid b\in\borel,\,X\in\bigoplus\limits_{\alpha\in R}\lgl^\alpha\right\}$$
where $R=\left\{\alpha\in\proots\mid\alpha\not\in\Roots_P,\,w(\alpha)>0\right\}$.
\end{prop}
\begin{proof}
The tangent space of $G/P$ at identity is $\lgl/\mathfrak p$. 
Consider the action of $P$ on $\lgl/\mathfrak p$ induced from the adjoint action of $P$ on $\lgl$.
The tangent bundle $T\,\gl/\para$ is the fiber bundle over $\gl/\para$ associated to the principal \para-bundle $\gl\rightarrow\gl/\para$, for the aforementioned action of \para\ on $\lgl/\mathfrak p$
, i.e., $T\,\gl/\para=\gl\times^\para\lgl/\mathfrak p$.
\\
\\
Let $R'=\left\{\alpha\in\roots^-\mid w(\alpha)>0\right\}$, 
so that $\proots=\Roots^+_P\sqcup R\sqcup -R'$.
Further, let $U_w=\left\langle U_\alpha\mid\alpha\in R'\right\rangle$.
For any point $b\in B$, we have (see, for example \cite{stein}):\begin{align*}
    BwP(\mod P) &=bBwP(\mod P)\\
                &=b(wU_ww^{-1})wP(\mod P)=bwU_wP(\mod P)
\end{align*}
It follows that the tangent subspace at $bw$ of the big cell $BwP(\mod P)$ is given by
\begin{align*}
    T_wBwP(\mod P)&=\left\{(bw,X)\in G\times^P\lgl/\mathfrak p\mid X\in\bigoplus\limits_{\alpha\in R'}\lgl^\alpha/\mathfrak p\right\}
\end{align*}
where 
$\lgl^\alpha/\mathfrak p$ denotes the image of a root space $\lgl^\alpha$ under the map $\lgl\rightarrow\lgl/\mathfrak p$.
Recall that the Killing form identifies the dual of a root space $\lgl^\alpha$ with the root space $\lgl^{-\alpha}$. 
Consequently, a root space $\lgl^\alpha\subset\u$ annihilates $T_{bw}BwP(\mod P)$ if and only if $\alpha\in\proots\backslash\Roots_P^+$ and $-\alpha\not\in R'$,
or equivalently, $\alpha\in R$.
The result now follows from the observation that $BwP(\mod P)$ is a dense open subset of $X_P(w)$, and is contained in the smooth locus of $X_P(w)$.
\end{proof}

\subsection{Schubert divisors}
A Schubert divisor in $G/P$ is a Schubert variety of codimension $1$.
Let $w_0^P$ be the longest element in $W^P$.
The affine permutation matrix for $w_0^P$ is given by\begin{align*}
    w_0^P=\begin{pmatrix}
    0&0&I(\lambda_r)\\
    0&{\iddots}&0\\
    I(\lambda_1)&0&0
    \end{pmatrix}
\end{align*}
where, for $k>0$, $I(k)$ denotes the $k\times k$ identity matrix.
Codimension one Schubert varieties $X_P(w)$ in $G/P$ correspond to $w=s_kw_0^P$, where $k=n-d_i$ for some $1\leq i<r$.
For $1\leq k<n$, let $v_k\in\What$ be given by the affine permutation matrix\begin{align*}
    v_k=\sum\limits_{i=1}^n a_iE_{i,n+1-i}  &&a_i=\begin{cases}t^{-1}&\text{if }i=k\\t&\text{if }i=k+1\\1&\text{otherwise}\end{cases}
\end{align*}
We denote by $v_k^\Para$ the minimal representative of $v_k$ with respect to $\What_\Para$.

\begin{prop}\label{vk}
Let $w=s_kw_0$, where $k=n-d_i$ for some $1\leq i<r$.
Then $X_\Para(v_k^\Para)$ is a compactification of $N^* X_P(w)$ via $\phi_P$. 
\end{prop}
\begin{proof}
We observe from \Cref{table} that $v_ks_i<v_k$ for all $\alpha_i\in S_P$, i.e., $v_k$ is maximal with respect to $W_P$. 
It follows that $v_k^\Para=v_kw_P$, where $w_P$ denotes the maximal element of $W_P$.
Next, we use \Cref{eq:length} to compute $l(v_k)=\dim\gl/B$. 
We now have 
$$l(v_k^\Para)=l(v_k)-l(w_P)=\dim G/B-\dim P=\dim G/P=\dim N^* X_P(w)$$
Since $X_\Para(\kappa)$ is irreducible, and has the same dimension as $N^* X_P(w)$, it suffices to show that $\phi_\para\left(N^* X_P(w)\right)\subset X_\Para(\kappa)$.
\\
\\
Let $L$ be the bottom left square sub-matrix of $w_0^P$ of size $d_{k-1}$ and $L'$ the top right square sub-matrix of $w_0^P$ of size $(n-d_{k+1})$.
We fix a lift $\overset\circ w$ of $w$ to the normalizer of $T$:\begin{align*}
\overset\circ w=\begin{pmatrix}
0&0&0&0&0&L'\\
0&0&0&I(\lambda_{k+1}-1)&0&0\\
0&e&0&0&0&0\\
0&0&0&0&1&0\\
0&0&I(\lambda_k-1)&0&0&0\\
L&0&0&0&0&0
\end{pmatrix}
\end{align*}
where $I(k)$ denotes the $k\times k$ identity matrix, and $e=\pm 1$ is determined by the equation $\det\overset\circ w=1$.
Observe that the $e$ is in the $(n-d_k,d_{k-1}+1)$ position and the $1$ in the $(n-d_k+1,d_{k+1})$ position.
Consider the root 
$$\gamma:=\sum\limits_{d_{i-1}<j<d_{i+1}}\alpha_j$$
Under the identification of \Cref{rootSystemA}, we have $\gamma=(d_{i-1}+1,d_{i+1})$.
We check that 
$$\left\{\alpha\in\proots\mid\alpha\not\in\Roots_P,\,w(\alpha)>0\right\}=\left\{\gamma\right\}$$
In particular, we can write a generic point of $N^* X_P(w)$ as $(b\overset\circ w,aE_\gamma)$. 
We may further assume $a\neq 0$.
It is now sufficient to show that 
$$\phi_P(b\overset\circ w,aE_\gamma)=b\overset\circ w\left(1-at^{-1}E_\gamma\right)\in X_\Para(v_k^\Para)$$
Consider $b_1,b_2,b_3\in\Borel$ given by 
\begin{align*}
    b_2 &=I(n)+\dfrac {et}a E_{k+1,k}           &b_3 =I(n)+\dfrac taE_{d_{i+1},d_{i-1}+1}\\
    b_1 &=\sum\limits_{1\leq i\leq n}c_i E_{ii} &c_i=\begin{cases}\frac e a&\text{for }i=k\\ ea&\text{for }i=k+1\\1&\text{otherwise}\end{cases}
\end{align*}
It is easily verified that $(b_1b_2b^{-1})b\overset\circ w(1-t^{-1}aE_\gamma)b_3$ is an affine permutation matrix corresponding to $v_k^\Para\in\What$.
\end{proof}

\begin{cor}
\label{frob}
Let $X_P(w)$ be a Schubert divisor in $G/P$.
The conormal variety $N^* X_P(w)$ is normal, Cohen-Macaulay, and Frobenius split.
\end{cor}
\begin{proof}
Schubert varieties in $\gL/\Para$ are normal, Cohen-Macaulay, and Frobenius split (cf. \cite{gf,mr}). 
Therefore, the same is true for $N^* X_P(w)$, since it is an open subset of $X_\Para(v_k^\Para)$.
\end{proof}



\begingroup
\small
\begin{thebibliography}{1}

\bibitem[AH]{ah} P. Achar and A. Henderson {\it Geometric Satake, Springer correspondence, and small representations}, Selecta Math. {\bf 19} (2013), 949--986.



\bibitem[Fu]{fu} B. Fu {\em Symplectic resolutions for nilpotent orbits}, Inventiones mathematicae, January 2003, Volume 151, Issue 1, pp 167--186.

\bibitem[Fa]{gf} G. Faltings {\em Algebraic Loop Groups and Moduli Spaces of Bundles} J. Eur. Math. Soc. 5, 41--68 (2003)

\bibitem[H]{he}  W. Hesselink {\em Polarizations in the Classical Groups}, Mathematische Zeitschrift Vol. 160, (1978) 217--234.

\bibitem[Ka]{vk} V. Kac {\em Infinite Dimensional Lie Algebras}, Third Edition. Cambridge University Press, 1990.

\bibitem[Ku]{sk} S. Kumar {\em Kac-Moody groups, their Flag Varieties and Representation Theory}, Prog. in Math. Vol 204, Birkh\"auser, 2002.

\bibitem[La]{vl} V. Lakshmibai {\em Cotangent bundle to the Grassmann Variety}, Transformation Groups Vol. 21, Issue 2, (2016) pp 519--530.



\bibitem[LSe]{gp2} V. Lakshmibai and C. S. Seshadri {\em Geomtery of G/P - II}, Proc. Ind. Acad. Sci. 87A (1978), 1--54.

\bibitem[LSS]{crv} V. Lakshmibai, C.S. Seshadri and R. Singh {\em Cotangent bundle to the Flag variety -I}, to appear in Transformation Groups.

\bibitem[LSi]{rv} V. Lakshmibai and R. Singh {\em Conormal varieties on the Cominuscule Grassmannian}

\bibitem[Lu1]{gl:green} G. Lusztig, {\it Green Polynomials and Singularities of Unipotent Classes}, Advances in Mathematics, vol. 42, Issue 2, November 1981, 169--178.

\bibitem[Lu2]{gl} G. Lusztig, {\it Canonical Bases arising from Quantized Enveloping Algebras}, Journal of the American Math. Soc. Vol. 3, No. 2, April 1990

\bibitem[MR]{mr} V. Mehta and A. Ramanathan {\em Frobenius splitting and cohomology vanishing for Schubert varieties}, Annals of Mathematics, 122 (1985), 27--40

\bibitem[MV]{mv} I. Mirkovi\'c and M. Vybornov {\em Quiver Varieties and Beilinson-Drinfeld Grassmannians of Type A} arXiv:0712.4160v2

\bibitem[R]{remy} B. Remy {\em Groupes de Kac-Moody D\'eploy\'es et Presque D\'eploy\'es}, Ast\'erisque 277 (2002), Soc. Math. Fr.

\bibitem[St]{stein} R. Steinberg {\em Lecture on Chevalley Groups} University Lecture Series, Volume 66, 2016
\bibitem[S]{strickland} E. Strickland {\em On the conormal bundle of the determinantal variety}, J. Algebra, vol 75 (1982), 523--537.

\end{thebibliography}
\endgroup
\end{document}